\documentclass[runningheads, envcountsame]{llncs}

\usepackage[normalem]{ulem}

\usepackage{amsmath,amssymb,amsfonts}
\usepackage{pdfpages}
\usepackage{url}
\usepackage{algorithmicx, algorithm, algpseudocode}

\spnewtheorem{fact}[theorem]{Fact}{\bfseries}{\itshape}
\spnewtheorem{clm}[theorem]{Claim}{\bfseries}{\itshape}
\spnewtheorem{conj}[theorem]{Conjecture}{\bfseries}{\itshape}
\spnewtheorem{rem}[theorem]{Remark}{\bfseries}{\rm}

\newcommand{\eps}{\epsilon}

\begin{document}

\title{Approximate Counting of Matchings in  $(3,3)$-Hypergraphs\thanks{Part of research of the 3rd and 4th authors done at Emory University, Atlanta and  another part during their visits to the Institut Mittag-Leffler (Djursholm, Sweden).} }
\titlerunning{\ }
\authorrunning{\ }

\author
{ Andrzej Dudek\inst{1}\thanks{Research supported by Simons
Foundation Grant \#244712 and by a grant from the Faculty Research
and Creative Activities Award (FRACAA), Western Michigan
University.} \and Marek Karpinski\inst{2}\thanks{Research
supported by DFG grants and the Hausdorff grant EXC59-1.} \and
Andrzej Ruci\'{n}ski\inst{3}\thanks{Research supported by the
Polish NSC grant N201 604 940  and the NSF grant DMS-1102086. }
\and Edyta Szyma\'{n}ska\inst{3}\thanks{Research supported by the
Polish NSC grant N206 565 740. } }

\institute{ Western Michigan University, Kalamazoo, MI, USA,
andrzej.dudek@wmich.edu \and Department of Computer Science,
University of Bonn, Germany, marek@cs.uni-bonn.de \and Faculty of
Mathematics and Computer Science, Adam Mic\-kie\-wicz University,
Pozna\'{n}, Poland, (rucinski,edka)@amu.edu.pl }

\maketitle

\begin{abstract}
We design a fully polynomial time approximation scheme (FPTAS) for
counting the number of matchings (packings) in arbitrary 3-uniform
hypergraphs of maximum degree three, referred to as
$(3,3)$-hypergraphs. It is the first polynomial time approximation
scheme for that problem, which includes also, as a special case,
the 3D Matching counting problem for 3-partite
$(3,3)$-hypergraphs.
 The proof technique of this paper
uses the general correlation decay technique and a new combinatorial
analysis of the underlying structures of the intersection
graphs. The proof method could be also of independent interest.

\end{abstract}

\section{Introduction}\label{intro}

The computational status of approximate counting of
matchings in hypergraphs has been open for some
time now, contrary to the existence of polynomial time
approximation schemes for graphs. The matching (packing)
counting problems in hypergraphs occur naturally in
the higher dimensional free energy problems, like in
the monomer-trimer systems discussed, e.g, by Heilmann \cite{Heil}.
The corresponding optimization versions of hypergraph
matching problem relate also to various allocations problems.

This paper aims at shedding some light on the approximation
complexity of that problem in 3-uniform hypergraphs of maximum
vertex degree three (called $(3,3)$-hypergraphs or $(3,3)$-graphs
for short). This class of hypergraphs includes also so-called 3D
hypergraphs, that is, (3,3)-graphs that are 3-partite. In
\cite{KRS-analco}, based on a generalization of the canonical
path method of Jerrum and Sinclair \cite{js}, we established a
fully polynomial time randomized approximation scheme (FPRAS) for
counting matchings in the classes of $k$-uniform hypergraphs
without structures called 3-combs. However, the status of the
problem  in arbitrary $(3,3)$-graphs was left wide open among with
other general problems for 3-, 4- and 5-uniform hypergraphs (for
$k\geq 6$ it is known to be hard, see Sec.~\ref{prelim}). In
particular, the existence of an FPRAS for counting matchings in $(3,3)$-graphs was
unknown.

In this paper we design the first fully polynomial time
approximation scheme (FPTAS) for arbitrary $(3,3)$-graphs.
The method of solution depends on the general correlation
decay technique and some new structural analysis of
underlying intersections graphs based on an extension
of the classical claw-freeness notion. The proof method used
in the analysis of our algorithm could be also of independent
interest.

The paper is organized as follows. Section~\ref{prelim} contains
some basic notions and preparatory discussions.  In
Sec.~\ref{main-thm} we formulate our main results and provide the
proofs. Finally, Sec.~\ref{disc} is devoted to the summary and an
outlook for future research.

\section{Preliminaries}\label{prelim}

 \emph{A hypergraph} $H=(V,E)$ is a finite set of vertices $V$ together with a family $E$ of
distinct, nonempty subsets of vertices called edges. In this paper
we consider $k$-\emph{uniform hypergraphs} (called further
\emph{$k$-graphs}) in which, for a fixed $k\ge2$,  each edge is of
size $k$.   \emph{A matching} in a hypergraph is a set (possibly
empty) of disjoint edges.

Counting  matchings is a \#P-complete problem already for  graphs ($k=2$) as proved by Valiant \cite{valiant}. In view of this hardness barrier, researchers turned to approximate counting, which initially has been accomplished via probabilistic techniques.

 Given a function   $C$ and a random variable $Y$ (defined on some probability space), and given two real numbers $\eps,\delta>0$, we say that $Y$ is an $(\eps,
\delta)$-\emph{approximation}  of $C$ if the probability  $\mathbb{P}\left(|Y(x)-C(x)|\ge\eps C(x)\right)\le\delta.$
 A \emph{fully polynomial
randomized approximation scheme} (FPRAS) for a function $f$ on
$\{0,1\}^*$ is a randomized algorithm which, for every triple
$(\eps,\delta, x)$, with $\eps>0,\;\delta>0$, and $x\in \{0,1\}^*,$
returns an $(\eps, \delta)$-approximation
$Y$ of $f(x)$ and runs in time polynomial in  $1/\eps$,
$\log(1/\delta)$, and $|x|$.

In this paper we investigate the problem of
counting the number of matchings in hypergraphs and try to determine the status of this problem
  for $k$-graphs with bounded degrees.

Let $deg_H(v)$ be the degree of vertex $v$ in a hypergraph $H$, that is, the number of edges of
$H$ containing  $v$. We denote by $\Delta(H)$  the maximum of $deg_H(v)$ over all $v$ in $H$. We call a $k$-graph $H$ a \emph{ $(k,r)$-graph} if $\Delta(H)\le r$.  Let $\#M(k,r)$ be the problem of counting the number of matchings in  $(k,r)$-graphs.

 Our inspiration comes from new results (both positive and negative) that emerged for approximate counting of the number of independent sets in graphs with bounded degree  and shed some light on the problem $\#M(k,r)$.

Let $\#IS(d)$ [$\#IS(\le d)$] be the problem of counting the number of all independent sets in $d$-regular graphs [graphs of maximum degree bounded by $d$, that is, $(2,d)$-graphs].
Luby and Vigoda \cite{lv} established an FPRAS for $\#IS(\le 4)$. This was complemented
later by the approximation hardness results for the   higher degree instances by Dyer, Frieze and Jerrum \cite{DFJ}.
The subsequent progress has coincided with  the  revival of a deterministic technique -- the spatial correlation decay method -- based on early papers of Dobrushin \cite{dobru} and Kelly \cite{kelly}.
It resulted in constructing deterministic  approximation schemes  for counting
independent sets in several classes of graphs with  degree (and other) restrictions, as well as for  counting matchings in  graphs of bounded
degree.

\begin{definition} A fully polynomial
time approximation scheme ({\bf FPTAS}) for a function $f$ on
$\{0,1\}^*$ is a deterministic algorithm which for every pair
$(\eps, x)$ with $\eps>0$, and $x\in \{0,1\}^*,$ returns a number $y(x)$ such  that
$$|y(x)-f(x)|\le\eps f(x),$$ and runs in time polynomial in  $1/\eps$, and
$|x|$.
\end{definition}

 In 2007 Weitz \cite{wei} found an FPTAS for  $\#IS(\le5)$, while, more recently, Sly \cite{sly} and Sly and Sun \cite{ss} complemented Weitz's result by proving the approximation  hardness
 for $\#IS(6)$, that is, proving that unless NP=RP,  there exists no  FPRAS (and thus, no FPTAS) for $\#IS(6)$. By applying two reductions: from $\#IS(6)$ to $\#M(6,2)$ (taking the dual hypergraph of a 6-regular graph), and from  $\#M(k,2)$ to $\#IS(k)$ (taking the intersection graph of a $(k,2)$-graph) for $k=3,4,5$, we conclude that

 \begin{enumerate}
\item[(i)] (unless NP=RP) there exists no FPRAS for $\#M(6,2)$;

\item[(ii)] there is an FPTAS for $\#M(k,2)$ with  $k\in\{3,4,5\}$.
\end{enumerate}

 Note  that the first reduction results, in fact, in a \emph{linear} $(6,2)$-graph, so the class of hypergraphs in question is even narrower.
 (A hypergraph is called
\emph{linear} when no two edges share more
than one vertex.)
 On the other hand, by the same kind of reduction it follows from a result of Greenhill \cite{green} that \emph{exact} counting of matchings is \#P-complete already in the class of   linear $(3,2)$-graphs.

 Facts (i) and (ii) above imply  that the only interesting cases for the positive results are those for $(k,r)$-graphs with $k=3,4,5$ and $r\ge3$, and thus, the smallest one among them is that of $(3,3)$-graphs.
 Our main result establishes an FPTAS for counting the  number of matchings in this class of hypergraphs.

\section{Main Result and the Proof}\label{main-thm}

The following theorem is the main result of this paper.

 \begin{theorem}\label{main}
The algorithm {\sl CountMatchings} given in Section~\ref{p-lem}
provides an FPTAS for $\#M(3,3)$ and runs in time
$O\left(n^2(n/\eps)^{\log_{50/49}144}\right)$.
\end{theorem}

\emph{The intersection graph}  of a hypergraph $H$ is the graph
$G=L(H)$ with vertex set $V(G)=E(H)$ and  edge set $E(G)$
consisting of all intersecting pairs of edges of $H$. When $H$ is
a graph, the intersection graph $L(H)$ is called \emph{the line
graph} of~$H$. Graphs which are line graphs of some graphs are
characterized by 9 forbidden induced subgraphs \cite{beineke}, one of which is
the \emph{claw}, an induced copy of $K_{1,3}$. There  is no
similar characterization for intersection graphs of $k$-graphs.
Still, it is easy to observe that for any $k$-graph $H$, its
intersection graph $L(H)$  does not contain an induced copy of
$K_{1,{k+1}}$. We shall call such graphs \emph{$(k+1)$-claw-free}.

Our proof of Thm.~\ref{main} begins with an obvious observation
that counting the number of matchings in a hypergraph $H$ is
equivalent to counting the number of independent sets in the
intersection graph $G=L(H)$. More precisely, let $Z_M(H)$ be the
number of matchings in a hypergraph $H$ and, for a graph $G$, let
$Z_I(G)$ be the number of independent sets in $G$. (Note that both
quantities  count the empty set.) Then $Z_M(H)=Z_I(L(H))$.

To approximately count the number of independent sets in a graph
$G=L(H)$ for a $(3,3)$-graph $H$, we apply some of the ideas from
\cite{bgknt} (the preliminary version of this paper appeared in
\cite{bgknt-STOC}) and \cite{suka}. In \cite{bgknt}  two new
instances of FPTAS were constructed, both based on the spatial
correlation decay method. First, for $\#M(2,r)$ with any given
$r$. Then, still in \cite{bgknt}, the authors refined their
approach to yield an FPTAS for counting independent sets in
claw-free graphs of bounded clique number which contain so called
\emph{simplicial cliques}. The last restriction  has been removed
by an ingenious observation in \cite{suka}.

 Papers \cite{bgknt,suka} inspired us to seek 
  adequate methods for  $(3,3)$-graphs. Indeed, for every $(3,3)$-graph $H$ its intersection graph $G=L(H)$ is 4-claw-free and has $\Delta(G)\le 6$. This turned out to be the right approach, as we deduced our
  Thm.~\ref{main} from  a technical lemma (Lem.~\ref{is-in-K14partialgraphs} below) which constructs an FPTAS for the number of independent sets in $K_{1,4}$-free graphs $G$ with $\Delta(G)\le 6$ and an additional property stemming from their being intersection graphs of $(3,3)$-graphs.



\subsection{Proof of Theorem \ref{main} -- Sketch and Preliminaries}\label{p-thm}

We deduce Thm.~\ref{main} from a technical lemma. The assumptions
of this lemma reflect some properties of the intersection graphs
of $(3,3)$-graphs.

\begin{lemma}\label{is-in-K14partialgraphs}
There exists an FPTAS for  the problem of counting independent
sets in every 4-claw-free graph with maximum degree at most 6 and such that the neighborhood of every vertex of  degree $d\ge5$ induces a subgraph that spans a matching of size $\lfloor d/2\rfloor$.
\end{lemma}

\begin{proof}[of Thm.~\ref{main}]
Given a $(3,3)$-graph $H$, consider its intersection graph~$G$.
Then $G$ is 4-claw-free, has maximum degree at most 6 and every
vertex neighborhood of size $d\ge5$ must span in $G$ a matching of
size $\lfloor d/2\rfloor$. This means that
Lem.~\ref{is-in-K14partialgraphs} applies to $G$ and there is an
FPTAS for counting independent sets of $G$ which is the same as
counting matchings in $H$. \qed
\end{proof}


It remains to prove Lem.~\ref{is-in-K14partialgraphs}. We begin
with underlining some properties of 4-claw-free graphs which are
relevant for our method. First, we introduce the notion of a
\emph{simplicial 2-clique} which is a generalization of a
simplicial clique introduced in \cite{CS} and utilized in
\cite{bgknt}. Throughout we assume notation $A\setminus B$ for set
differences and, for $A\subset V(G)$, we write $G-A$ for the graph
operation of deleting from $G$ all vertices belonging to $A$. In
other words, $G-A=G[V(G)\setminus A]$. Also, for any graph $G$, we
use $\delta(G)$ to denote its minimum vertex degree and
$\alpha(G)$  for the size of the largest independent set in $G$.

\begin{definition}A set $K\subseteq V(G)$ is a \emph{2-clique} if $\alpha(G[K])\leq 2.$ A 2-clique is \emph{simplicial} if
for every $v\in K,\, N_G(v)\setminus K$ is a 2-clique in $G-K$.
\end{definition}

For us a crucial property of simplicial 2-cliques is that if $G$ is a connected 
graph containing a nonempty simplicial 2-clique $K$ then it is easy to find another simplicial
2-clique in the induced subgraph $G-K$, and consequently, the whole vertex set of $G$ can be partitioned into blocks which are simplicial 2-cliques in suitable nested sequence of induced subgraphs of $G$ (see Claim \ref{block2}).

However, in the proof of Lem.~\ref{is-in-K14partialgraphs} we
shall use a special class of 2-cliques.

\begin{definition} A 2-clique $K$ in a graph $G$ is called \emph{a block} if $|K|\le4$ and  $\delta(G[K])\ge1$ whenever $|K|=4$. A block $K$ is \emph{simplicial} if for every $v\in K$ the set $N_G(v)\setminus K$ is a block in $G-K$.
\end{definition}

Next, we state a trivial but useful observation which follows straight from the above definition. (We consider the empty set as a block too.)

\begin{fact}\label{indu} If $K$ is a (simplicial) block in $G$ then for every $V'\subseteq  V(G)$ the set $K\cap V'$ is a (simplicial) block in the induced subgraph $G[V']$ of $G$.
\end{fact}


Let  a graph $G$ satisfy the assumptions of
Lem.~\ref{is-in-K14partialgraphs}. The next claim provides a
vital, ``self-reproducing'' property of blocks in $G$.

\begin{clm}\label{block} If $K$ is a simplicial block in
$G$, then for every $v\in K$ the set $N_G(v)\setminus K$ is a simplicial block in
$G-K.$
\end{clm}

\begin{proof}
Set $K_v:=N_G(v)\setminus K$ for convenience. By definition of $K$, $K_v$ is a  block. It remains to show that $K_v$ is simplicial.
Let $u\in K_v$ and let $K_u=N_G(u)\setminus(K\cup N_G(v))$. Suppose there is an independent set $I$ in $G[K_u]$ of size $|I|=3$. Then $u,v$ and the vertices of $I$ would form an induced $K_{1,4}$ in $G$ with $u$ in the center. As this is a contradiction, we  conclude that $K_u$ is a 2-clique.

To show that $K_u$ is indeed a block, note first that, by the assumptions that $\Delta(G)\le 6$, we have $|K_u|\le 5$. However, if $|K_u|=5$ then $v$ would be an isolated vertex in $G[N_G(u)]$. But $G[N_G(u)]$ spans a matching of size~3 since $\delta(u)=6$ -- a contradiction. For the same reason, if $|K_u|=4$ then regardless of the degree of $u$ in $G$ (which might be 5 or 6) there can be no isolated vertex in $G[K_u]$, since $G[K_u]$ must span a matching of size~2. 
\qed
\end{proof}

Our next claim asserts that once there is a nonempty block in $G$, one can find a suitable partition of $V(G)$ into sets which are blocks in a nested  sequence of induced subgraphs of $G$ defined by deleting these sets one after another.

\begin{clm}\label{block2} Let $K$ be a nonempty simplicial block in $G$. If, in addition, $G$ is connected then there exists a partition $V(G)=K_1\cup\cdots\cup K_m$ such that $K_1=K$ and for every $i=2,\dots,m$, $K_i$ is a nonempty, simplicial block in $G_i:=G-\bigcup_{j=1}^{i-1}K_j$.
\end{clm}

\begin{proof}
Suppose we have already constructed disjoint sets $K_1\cup\cdots \cup K_s$, for some $s\ge1$, such that  $K_1=K$, for every $i=2,\dots,s$, $K_i$ is a nonempty, simplicial block in $G_i:=G-\bigcup_{j=1}^{i-1}K_j$, and that $R_s:=V(G)\setminus \bigcup_{i=1}^{s}K_{i}\neq\emptyset$. Since $G$ is connected, there is an edge between a vertex in $R_s$ and a vertex $v\in K_i$ for some $1\le i\le s$. Since $K_i$ is a simplicial  block in $G_i$, by Fact \ref{indu}, it is also simplicial in its subgraph $G_i[V']$, where $V'=K_i\cup R_s$, that is the subgraph of $G_i$ obtained by deleting all vertices of $K_{i+1}\cup \cdots\cup K_{s}$. 
Now apply Claim \ref{block} to $G_i[V']$, $K_i$,  and $v$, to conclude that $N_G(v)\cap R_s$ is a simplicial block in $G_{s+1}:=G-\bigcup_{i=1}^{s}K_i$.
\qed
\end{proof}

Let $K_1,K_2,\ldots, K_m$ be as in Claim \ref{block2}.
Then,

\begin{equation}\label{tele}
Z_I(G)=\frac{Z_I(G_1)}{Z_I(G_2)}\cdot \frac{Z_I(G_2)}{Z_I(G_3)}\cdot\ldots\cdot\frac{Z_I(G_i)}{Z_I(G_{i+1})}\cdot\ldots\cdot \frac{Z_I(G_{m})}{Z_I(G_{m+1})},
\end{equation}
where $G_{m+1}=\emptyset$ and $Z_I(G_{m+1})=1$.
Observe that for each $i$, $G_{i+1}=G_i-K_i$ and the reciprocal of each quotient in (\ref{tele}) is precisely the probability
\begin{equation}\label{prob}
\mathbb{P}_{G_i}(K_i\cap
\mathbf{I}=\emptyset)=\frac{Z_I(G_{i}-K_i)}{Z_I(G_i)},
\end{equation}
 where
$\mathbf{I}$ is an independent set of $G_i$ chosen uniformly at
random. In view of this, the main step in building an FPTAS for
$Z_I(G)$ will be to approximate the probability
$\mathbb{P}_{G}(K_i\cap \mathbf{I}=\emptyset)$ within $1\pm\tfrac
{\eps}{n}$ (see Sec.~\ref{p-lem} and Algorithm~\ref{countis}
therein).

But what if $G$ is disconnected or does not contain a simplicial block to start with?
First, if $G=\bigcup_{i=1}^c G_i$ consists of $c$ connected components $G_1,\dots,G_c$, then, clearly
\begin{equation}\label{mnoz}
Z_I(G)=\prod_{i=1}^cZ_I(G_i)
\end{equation}
 and the problem reduces to that for connected graphs.

As for the second obstacle, Fadnavis \cite{suka} proposed a very
clever observation to cope with it. Let $G$ be a connected graph
satisfying the assumptions of Lem.~\ref{is-in-K14partialgraphs}
and let $v\in V(G)$ be such that $G-v$ is connected. By
considering the fate of vertex $v$, we obtain the recurrence
\begin{equation}\label{rec}
Z_I(G)=Z_I(G-v)+Z_I(G^v),
\end{equation}
where $G^v=G-N_G[v]$ and $N_G[v]=N_G(v)\cup\{v\}$. Let $G^v=\bigcup_{i=1}^cG_i^v$ be the partition of $G^v$ into its connected components. For each $i$ let $u_i\in N_G(v)$ be such that $N_G(u_i)\cap V(G_i^v)\neq\emptyset$. Owing to the connectedness of $G-v$,  a vertex $u_i$ must exist. Set $K_i=N_G(u_i)\cap V(G_i^v)$.

\begin{clm}\label{block3}
The set $K_i$ is a simplicial block in $G_i^v$.
\end{clm}

\begin{proof}
The proof is quite similar to that of Claim \ref{block}. We first prove that $K_i$ is a block. Suppose there is an independent set $I$ in $G[K_i]$ of size $|I|=3$. Then $u_i,v$ and the vertices of $I$ would form an induced $K_{1,4}$ in $G$ with $u_i$ in the center. As this is a contradiction, we  conclude that $K_i$ is a 2-clique. To prove that $K_i$ is, in fact, a block, notice that there is no edge between $v$ and $K_i$. Thus, we cannot have $|K_i|=5$ because then $v$ would be an isolated vertex in $G[N(u_i)]$ -- a contradiction with the assumption on $G$. If, however, $|K_i|=4$ then $v$ is the (only) isolated vertex in $G[N(u_i)]$ and, consequently, $\delta(G[K_i])\ge1$.

It remains to show that the block $K_i$ is simplicial, that is,
for every $w\in K_i$, the set $N_{G_i^v}(w)\setminus K_i$ is a
block in $G_i^v-K_i$. This, however, can be proved mutatis
mutandis as in the proof of Claim~\ref{block}.\qed
\end{proof}


For the first term of  recurrence (\ref{rec}) we apply (\ref{rec}) recursively. In view of Claim \ref{block3}, to the second term of recurrence (\ref{rec}) one can apply formula (\ref{mnoz}) and then each term $Z_I(G_i^v)$ can be approximated based on~(\ref{tele})~and~(\ref{prob}).

\subsection{The Remainder of the Proof of Lemma~\ref{is-in-K14partialgraphs}} \label{p-lem}

Hence, it remains to approximate $\mathbb{P}_{G}(K\cap
\mathbf{I}=\emptyset)=\tfrac{Z_I(G-K)}{Z_I(G)}$ within $1\pm\tfrac{\eps}n$, where $K$ is a simplicial block in $G$.
We set $N_v:=N_G(v)$ and  formulate the following recurrence
relation by considering how an independent set may intersect $K$:

$$Z_I(G)=Z_I(G-K)+\sum_{v\in K}Z_I(G-(N_v\cup K))+\frac12\sum_{uv\notin G[K]}Z_I(G-(N_u\cup N_v\cup
K))$$ or equivalently, after dividing sidewise by $Z_I(G-K)$,
$$\frac{Z_I(G)}{Z_I(G-K)}=1+\sum_{v\in
K}\frac{Z_I(G-(N_v\cup K))}{Z_I(G-K)}+\frac12\sum_{uv\notin
G[K]}\frac{Z_I(G-(N_u\cup N_v\cup K))}{Z_I(G-K)}.$$

Here and throughout the inner summation ranges over all \emph{ordered} pairs of  \emph{distinct} vertices of $K$ such that $\{u,v\}\notin G[K]$.
 At this point, in view of symmetry, it seems redundant to consider ordered pairs (and consequently have the factor of $\tfrac12$ in front of the sum), but we break the symmetry right now as we further observe that

$$\frac{Z_I(G-(N_u\cup N_v\cup K))}{Z_I(G-K)}=\frac{Z_I(G-(N_u\cup N_v\cup
K))}{Z_I(G-(N_v\cup K))}\cdot \frac{Z_I(G-(N_v\cup K))}{Z_I(G-K)}.$$

By Claim \ref{block}, $N_v\setminus K$ is a simplicial block in $G-K$. We need to show that, similarly, $N_u\setminus(N_v\cup K)$ is a simplicial block in $G-(N_v\cup K)$.

\begin{clm}\label{block4} Let $K$ be a simplicial
block in  $G$ and let $u,v\in K$ be such that $u\neq v$ and $uv\notin
G[K]$. Further, let $H:=G-(N_G(v)\cup K).$ Then $N_H(u)$ is a simplicial
block in $H.$
\end{clm}
\begin{proof} By Claim \ref{block}, the set $N_G(u)\setminus K$ is a simplicial block in $G-K$. Apply Fact \ref{indu} to  $N_G(u)\setminus K$ and $G-K$ with $V'=V(H)$.
\qed
\end{proof}

Let
$$\Pi_G(K):={\mathbb P}(K\cap \mathbf{I}=\emptyset)=\frac{Z_I(G-K)}{Z_I(G)},$$
where $\mathbf{I}$ is a random independent set of $G$.
Finally, setting $K_v:=N_v\setminus K$ and $K_{uv}:=N_u\setminus(N_v\cup K)$,  and rewriting $G-(N_v\cup K)=G-K-K_v$, we get the recurrence for the probabilities:
\begin{equation*}
{\Pi}^{-1}_{G}(K)=1+\sum\limits_{v\in
K}{\Pi}_{G-K}(K_v)\left(1+\frac12\sum\limits_{uv\notin G[K]}{\Pi}_{G-K-K_v}(K_{uv})\right).
\end{equation*}

This recurrence, in principle, allows one to compute ${\Pi}_{G}(K)$ exactly, but only in an exponential number of steps.
Instead,
we will approximate it by a function $\Phi_G(K,t)$, also defined recursively, which ``mimics'' $\Pi_G(K)$ but has a built-in time counter $t$.

\begin{definition}\label{rec-phi} For every graph $G$, every simplicial block $K$ in $G$  and an integer  $t\in \mathbb{Z}_+$, the function $\Phi_G(K,t)$ is defined recursively as follows: $\Phi_G(K,0)=\Phi_G(K,1)=1$ as well as $\Phi_G(\emptyset,t)=1$,  while for $t\ge2$ and $K\neq\emptyset$
\begin{equation*}
\Phi^{-1}_G(K,t)=1+\sum\limits_{v\in K}\Phi_{G-K}(K_v,t-1)\left(1+\frac12\sum\limits_{uv\notin G[K]}\Phi_{G-K-K_v}(K_{uv},t-2)\right).
\end{equation*}
\end{definition}

Now we are ready to state the  algorithm {\sl CountMatchings} for computing $Z_M(H)$  for any connected $(3,3)$-graph $H$ and its subroutine {\sl CountIS} for computing $Z_I(G)$ in a subgraph of  $G=L(H)$ containing a simplicial block $K$.

\begin{algorithm}[h]
\caption{{\sl CountMatchings}($H,t$)}\label{countmatchalg}
\begin{algorithmic}[1]

\State $G:=L(H).$

\State $Z_M:=1,\, F:=G.$

\While {$F\neq\emptyset$ }

\State Pick $v\in V(F)$ s.t. $F-v$ is connected.
\State $F^v:=F-N_F[v]$
\State If $F^v=\emptyset$ then $Z_M=Z_M+1$ and go to Line 3.
\State $F^v=\bigcup_{i=1}^cF_i^v$, where $F_i^v$ are  connected components of $F^v$.
\For  {$i:=1$ to c }
\State {Find  $K_i$ as in Claim \ref{block3}}
\EndFor
\State $Z_M:= Z_M+\prod_{i=1}^c${\sl CountIS}$(F_i^v,K_i,t)$
\State $F:=F-v$
\EndWhile
\State Return $Z_M$
\end{algorithmic}
\end{algorithm}

\begin{algorithm}[h]
\caption{{\sl CountIS}($G,K,t$)}\label{countis}
\begin{algorithmic}[1]
\State Let $V(G)=\bigcup_{i=1}^mK_i$ be a partition of $V(G)$ as in Claim \ref{block2} with $K_1=K$.
\State $Z_I:=1, F:=G$
\For {$i=1$ to $m$}
\State $Z_I := \frac {Z_I}{\Phi_F(K_i,t)}$
\State $F := F-K_i$
\EndFor
\State Return $Z_I$
\end{algorithmic}
\end{algorithm}

We will show that already for $t=\Theta(\log n)$, when $\Phi$ can be easily computed in polynomial time, the two functions become close to each other.

Note that both quantities, ${\Pi}_{G}(K)$ and $\Phi_G(K,t)$, fall into the interval $[\tfrac1{9},1]$. The lower bound is due to the fact that a block has at most 4 vertices and each of them has degree at most 2 in $G^c$, so that the total number of terms in the denominator is at most nine, five of them do not exceed 1, while eight of them do not exceed $\tfrac12$. Our goal is to approximate ${\Pi}_{G}(K)$ by $\Phi_G(K,t)$, for a suitably chosen $t$, within the multiplicative factor of $1\pm \epsilon/n$. In view of the above lower bound, it suffices to show that
$|{\Pi}_{G}(K)-\Phi_G(K,t)|\le\frac{\epsilon}{9n}.$

To achieve this goal, we will use the correlation decay technique
which boils down to establishing a recursive bound on the above
difference (cf. \cite{bgknt}). The success of this method depends
on the right choice of a pair of functions $g$ and $h$, with
$g:[0,1]\to\Re$, such that they are inverses of each other, that
is, $g\circ h\equiv1$. Then we define a function $f_K$ of
$|K|+2e(G^c[K])$ variables, one for each vertex and each (ordered)
non-edge of $G[K]$, as follows. Let
$\bold{z}=(z_1,\dots,z_{|K|},z_{uv}:uv\notin G[K])$ be a vector of
variables of that function. For ease of notation, we denote the
set of all indices of the coordinates of function $f_K$ by $J$,
that is, we set $J:=K\cup\{(u,v):\{u,v\}\notin G[K]\}.$ Then
\begin{equation}\label{fkz}
f_K(\bold{z}):=f(\bold{z})=g\left(\left\{1+\sum\limits_{v\in
K}h(z_v)\left(1+\frac12\sum\limits_{uv\notin
G[K]}h(z_{uv})\right)\right\}^{-1}\right). \end{equation}
 To
understand the reason for this set-up, put $x:=g({\Pi}_{G}(K))$,
$x_v:=g({\Pi}_{G-K}(K_v))$, $ x_{uv}:=g({\Pi}_{G-K-K_v}(K_{uv})),$
and, correspondingly,
$$y:=g(\Phi_{G}(K,t))\quad y_v:=g(\Phi_{G-K}(K_v,t-1))\quad y_{uv}:=g(\Phi_{G-K-K_v}(K_{uv},t-2)).$$
 Then, $f(\bold{x})=x$ and
$f(\bold{y})=y$, and so the difference we are after can be
expressed as $|x-y|=|f(\bold{x})-f(\bold{y})|.$ Thus, we are in
position to apply the Mean Value Theorem to $f$ and conclude that
there exists $\alpha\in[0,1]$ such that, setting
$\bold{z_\alpha}=\alpha\bold{x}+(1-\alpha)\bold{y}$,
$$|f(\bold{x})-f(\bold{y})|=|\nabla f(\bold{z_\alpha})(\bold{x}-\bold{y})|\le|\nabla f(\bold{z_\alpha})|\times \max_{\kappa\in J}|x_{\kappa}-y_{\kappa}|.$$

It remains to  bound $\max_z|\nabla f(\bold{z})|$ from above,
uniformly by a constant $\gamma<1$. Then, after iterating at most
$t$ but at least $t/2$ times, we will arrive at a triple
$(G',K',t')$, where $G'$ is an induced subgraph of $G$, $K'$ is a
block in $G'$, and $t'\in\{0,1\}$. At this point, setting
$\mu_g:=|g(1)|+|\max_s g(s))|$, we will obtain the ultimate bound
$$|x-y|\le \gamma^{t/2}\times|g({\Pi}_{G'}(K'))-g(1)|\le \gamma^{t/2}\times \mu_g \le \frac{\epsilon}{9n},$$
\begin{equation}\label{t}\mbox{for}\qquad
t\ge2\log((9\mu_g n)/\eps)/\log(1/\gamma).
\end{equation}

In \cite{bgknt}, to estimate  $|\nabla f(\bold{z})|$ for a similar function $f$, the authors chose $g(s)=\log s$ and $h(s)=e^s$. This choice, however, does not work for us. Instead, we set
$g(s)=s^{1/4}$ and $h(s)=s^4.$ Then, $\mu_g=2$ and
$$|\nabla f(\bold{z})|\le\sum_{\kappa\in J}\Bigg|\frac{\partial f(\bold{z})}{\partial z_\kappa}\Bigg|=
\frac{\sum\limits_{v\in K}\left\{z_v^3+\frac12\sum\limits_{uv\notin G[K]}(z_v^3z_{uv}^4+z_v^4z_{uv}^3)\right\}}{\left\{1+\sum\limits_{v\in
K}z_v^4\left(1+\frac12\sum\limits_{uv\notin G[K]}z_{uv}^4\right)\right\}^{5/4}}.$$

Observe that $f_K$ depends only on the isomorphism type of $G[K]$, a graph on up to 4 vertices, with no independent set of size 3, and with no isolated vertex when $|K|=4$. Let us call all these graphs \emph{block graphs}. One block graph is given in Figure \ref{1} below.

\begin{figure}
\centering
\begin{picture}(0,0)%
\includegraphics{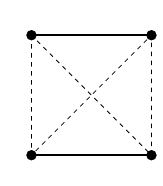}%
\end{picture}%
\setlength{\unitlength}{3158sp}%
\begingroup\makeatletter\ifx\SetFigFont\undefined%
\gdef\SetFigFont#1#2#3#4#5{%
  \reset@font\fontsize{#1}{#2pt}%
  \fontfamily{#3}\fontseries{#4}\fontshape{#5}%
  \selectfont}%
\fi\endgroup%
\begin{picture}(1605,1836)(-314,-664)
\put(301,-211){\makebox(0,0)[lb]{\smash{{\SetFigFont{10}{12.0}{\rmdefault}{\mddefault}{\updefault}{\color[rgb]{0,0,0}$z_{24}$}%
}}}}
\put(-149,989){\makebox(0,0)[lb]{\smash{{\SetFigFont{10}{12.0}{\rmdefault}{\mddefault}{\updefault}{\color[rgb]{0,0,0}$z_1$}%
}}}}
\put(-149,-586){\makebox(0,0)[lb]{\smash{{\SetFigFont{10}{12.0}{\rmdefault}{\mddefault}{\updefault}{\color[rgb]{0,0,0}$z_4$}%
}}}}
\put(1126,989){\makebox(0,0)[lb]{\smash{{\SetFigFont{10}{12.0}{\rmdefault}{\mddefault}{\updefault}{\color[rgb]{0,0,0}$z_2$}%
}}}}
\put(1126,-586){\makebox(0,0)[lb]{\smash{{\SetFigFont{10}{12.0}{\rmdefault}{\mddefault}{\updefault}{\color[rgb]{0,0,0}$z_3$}%
}}}}
\put(-299,164){\makebox(0,0)[lb]{\smash{{\SetFigFont{10}{12.0}{\rmdefault}{\mddefault}{\updefault}{\color[rgb]{0,0,0}$z_{14}$}%
}}}}
\put(1276,164){\makebox(0,0)[lb]{\smash{{\SetFigFont{10}{12.0}{\rmdefault}{\mddefault}{\updefault}{\color[rgb]{0,0,0}$z_{23}$}%
}}}}
\put(301,614){\makebox(0,0)[lb]{\smash{{\SetFigFont{10}{12.0}{\rmdefault}{\mddefault}{\updefault}{\color[rgb]{0,0,0}$z_{13}$}%
}}}}
\end{picture}%
\caption{The essential block graph.}\label{1}
\end{figure}
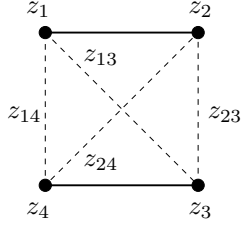

In  a sense we just need to consider this  one block graph. Indeed, the complement of every block graph is contained in the complement of the block graph in Figure \ref{1}. Hence, it suffices to maximize $|\nabla f(\bold{z})|$ just for this graph. Our computational task is, therefore, to bound from above

{\small{
\begin{align*}
F&(\mathbf{z}) = \| \nabla(\mathbf{z}) \|_{1} =
\frac{1}{4}
 \bigg{(}\!
1+z_1^4+z_2^4+z_3^4+z_4^4 +\\
&\qquad\qquad\qquad\quad \frac{1}{2}\Big{(} z_{14}^4\big{(} z_1^4+z_4^4\big{)} + z_{13}^4\big{(} z_1^4+z_3^4\big{)} +
z_{23}^4\big{(} z_2^4+z_3^4\big{)} + z_{24}^4\big{(} z_2^4+z_4^4\big{)}\!\Big{)}\!\!\bigg{)}^{-5/4} \times\\
&\bigg{(}\!
2z_1^3 \big{(} 2 + z_{14}^4 + z_{13}^4\big{)} +
2z_2^3 \big{(} 2 + z_{23}^4 + z_{24}^4\big{)} +
2z_3^3 \big{(} 2 + z_{13}^4 + z_{23}^4\big{)} +
2z_4^3 \big{(} 2 + z_{14}^4 + z_{24}^4\big{)} \ +\\
& \qquad\qquad\qquad\qquad\quad\ 2z_{14}^3\big{(} z_1^4+z_4^4\big{)} +
2z_{13}^3\big{(} z_1^4+z_3^4\big{)} +
2z_{23}^3\big{(} z_2^4+z_3^4\big{)} +
2z_{24}^3\big{(} z_2^4+z_4^4\big{)}\!\!
\bigg{)}.
\end{align*}
}}
One can show (using, e.g., \emph{Mathematica}) that $F(\mathbf{z})< 0.971$ for $0\le z_i\le 1$ and $0\le z_{ij}\le 1$. Thus, we have (\ref{t})
 with $\mu_g=2$ and, say,  $\gamma=0.98=\tfrac{49}{50}$. Summarizing, the running time of computing $\Phi_G(K,t)$ in Step 4 of Algorithm~\ref{countis} is ${12}^{t}$ since there at most 12 expressions
 to compute in each step of the recurrence relation (see Def.~\ref{rec-phi}). Also, {\sl CountIS} takes at most $|V(F_i^v)|12^t$ steps and hence, Line 11 of
 {\sl CountMatchings} takes $n{12}^t$ steps and is invoked at most $n$ times.
Consequently, with $t=2\lceil\log((18n)/\eps)/\log(50/49)\rceil$
we get the running time of our algorithm of order
$O\left(n^2(n/\eps)^{\log_{50/49}144}\right)$.

\begin{rem}
With basically the same proof we can construct an FPTAS for
calculating the partition function
$Z_M(H,\lambda)=\sum_{M}\lambda^{|{M}|},$ where the sum runs over
all matchings in $H$, for any constant $\lambda\in (0,1.077]$. The $\lambda$ factor will appear in front
of each summation in (\ref{fkz}), which one can neutralize by
setting $h(s)=\tfrac{s^4}{\lambda}$ and $g(s)=(\lambda s)^{1/4}.$
\end{rem}

\section{Summary, Discussion, and Further Research}\label{disc} The main result of this paper (Thm.~\ref{main}) establishes an FPTAS for the problem $\#M(3,3)$ of counting the number of matchings in a $(3,3)$-graph. A reformulation of Thm.~\ref{main} in terms of graphs yields an FPTAS for  the problem of counting independent
sets in every graph which is the intersection graph of a
$(3,3)$-graph. As mentioned earlier, every intersection graph of a
$(3,3)$-graph is 4-claw-free. Moreover, its maximum degree is at
most six. We wonder if there exists an FPTAS for  the problem of
counting independent sets in every 4-claw-free graph with maximum
degree at most 6. Lemma~\ref{is-in-K14partialgraphs} falls short
of proving that. The missing part is due to our inability to
repeat the above estimates for 2-cliques of size five.

In an earlier paper \cite{KRS-analco} three of the authors have
found an FPRAS for the number of matchings in $k$-graphs without
3-combs. As their intersection graphs are claw-free, it follows
from the above mentioned result  on independent sets in
\cite{bgknt,suka} that there is also an FPTAS  for the number of
matchings in $(k,r)$-graphs without 3-combs, for any fixed $r$. In
view of this conclusion and Thm.~\ref{main}, we raise the question
if for all $k\leq 5$ and $r$ there is an FPTAS (or at least FPRAS) for
the problem $\#M(k,r)$. The first open instance is that of
$(3,4)$-graphs. For $k=4, 5$, to avoid recurrences of depth
$k-1\ge3$, as an intermediate step, one could first consider  the
restriction of the class of $(k,r)$-graphs to those without a
4-comb, that is, to those whose intersection graphs are
4-claw-free. Here, the first open instance is that of
$(4,3)$-graphs without 4-combs. In general, it would be also very
interesting
   to elucidate the status
   of the problem for arbitrary $k$-graphs for $k = 3, 4$ and 5, or for
   some generic subclasses of them.

\section*{Acknowledgements}
We thank Martin Dyer and Mark Jerrum for stimulating discussions
on the subject of this paper and the referees for their valuable
comments. We are also very grateful to Michael Simkin who pointed out and fixed an error (cf. Lemma~\ref{is-in-K14partialgraphs} and the proof of Claim~\ref{block}) in an earlier version of this paper~\cite{DKRS}.

\begin{thebibliography}{10}
\providecommand{\url}[1]{\texttt{#1}}
\providecommand{\urlprefix}{URL }

\bibitem{bgknt-STOC}
Bayati, M., Gamarnik, D., Katz, D., Nair, C., Tetali, P.: Simple deterministic
  approximation algorithms for counting matchings. In:
  S{TOC}'07---{P}roceedings of the 39th {A}nnual {ACM} {S}ymposium on {T}heory
  of {C}omputing, pp. 122--127. ACM (2007)

\bibitem{bgknt}
Bayati, M., Gamarnik, D., Katz, D., Nair, C., Tetali, P.: Simple deterministic
  approximation algorithms for counting matchings (2008),
  \url{http://people.math.gatech.edu/~tetali/PUBLIS/BGKNT_final.pdf}

\bibitem{beineke}
Beineke, L.W.: Characterizations of derived graphs. J. Combin. Theory  9,
  129--135 (1970)

\bibitem{CS}
Chudnovsky, M., Seymour, P.: The roots of the independence polynomial of a
  clawfree graph. J. Combin. Theory Ser. B  97(3),  350--357 (2007)

\bibitem{dobru}
Dobrushin, R.: {Prescribing a system of random variables by conditional
  distributions.} Theor. Probab. Appl.  15,  458--486 (1970)

\bibitem{DKRS}
Dudek, A., Karpinski, M., Ruci\'nski, A., Szyma\'nska, E.: Approximate counting
  of matchings in {$(3,3)$}-hypergraphs. In: Algorithm theory---{SWAT} 2014,
  Lecture Notes in Comput. Sci., vol. 8503, pp. 380--391. Springer, Cham (2014)

\bibitem{DFJ}
Dyer, M., Frieze, A., Jerrum, M.: On counting independent sets in sparse
  graphs. SIAM J. Comput.  31(5),  1527--1541 (2002)

\bibitem{suka}
Fadnavis, S.: Approximating independence polynomials of claw-free graphs
  (2012), \url{http://www.math.harvard.edu/~sukhada/IndependencePolynomial.pdf}

\bibitem{green}
Greenhill, C.: The complexity of counting colourings and independent sets in
  sparse graphs and hypergraphs. Comput. Complexity  9(1),  52--72 (2000)

\bibitem{Heil}
Heilmann, O.: Existence of phase transitions in certain lattice gases with
  repulsive potential. Lett. Al Nuovo Cimento Series 2  3(3),  95--98 (1972)

\bibitem{js}
Jerrum, M., Sinclair, A.: Approximating the permanent. SIAM J. Comput.  18(6),
  1149--1178 (1989)

\bibitem{KRS-analco}
Karpi\'{n}ski, M., Ruci\'{n}ski, A., Szyma\'{n}ska, E.: Approximate counting of
  matchings in sparse uniform hypergraphs. In: 2013 {P}roceedings of the
  {W}orkshop on {A}nalytic {A}lgorithmics and {C}ombinatorics ({ANALCO}), pp.
  72--79. SIAM (2013)

\bibitem{kelly}
Kelly, F.P.: Stochastic models of computer communication systems. J. Roy.
  Statist. Soc. Ser. B  47(3),  379--395, 415--428 (1985)

\bibitem{lv}
Luby, M., Vigoda, E.: Fast convergence of the {G}lauber dynamics for sampling
  independent sets. Random Structures Algorithms  15(3-4),  229--241 (1999)

\bibitem{sly}
Sly, A.: Computational transition at the uniqueness threshold. In: 2010 {IEEE}
  51st {A}nnual {S}ymposium on {F}oundations of {C}omputer {S}cience {FOCS}
  2010, pp. 287--296 (2010)

\bibitem{ss}
Sly, A., Sun, N.: The computational hardness of counting in two-spin models on
  d-regular graphs. In: FOCS, pp. 361--369 (2012),
  \url{http://arxiv.org/abs/1203.2602}

\bibitem{valiant}
Valiant, L.G.: The complexity of enumeration and reliability problems. SIAM J.
  Comput.  8(3),  410--421 (1979)

\bibitem{wei}
Weitz, D.: Counting independent sets up to the tree threshold. In: S{TOC}'06:
  {P}roceedings of the 38th {A}nnual {ACM} {S}ymposium on {T}heory of
  {C}omputing, pp. 140--149. ACM (2006)

\end{thebibliography}

\newpage

\section*{Appendix: Mathematica expressions}\label{appendix}

First we define $F(\mathbf{z})$ function:

{\small{
\begin{verbatim}
F[z1_, z2_, z3_, z4_, z14_, z13_, z23_, z24_]:=\
1/4(1+z1^4+z2^4+z3^4+z4^4+\
      1/2(z14^4(z1^4+z4^4)+z13^4(z1^4+z3^4)+\
      z23^4(z2^4+z3^4)+z24^4(z2^4+z4^4)))^(-5/4)\
  (2z1^3(2+z14^4+z13^4)+2z2^3(2+z23^4+z24^4)+\
   2z3^3(2+z13^4+z23^4)+2z4^3(2+z14^4+z24^4)+\
   2z14^3(z1^4+z4^4)+2z13^3(z1^4+z3^4)+\
   2z23^3(z2^4+z3^4)+2z24^3(z2^4+z4^4))
\end{verbatim}
}}

\noindent
Next we find the absolute maximum:

{\small{
\begin{verbatim}
NMaximize[{F[z1, z2, z3, z4, z14, z13, z23, z24],\
     0<=z1<=1 && 0<=z2<=1 && 0<=z3<=1 && 0<=z4<=1 &&
     0<=z14<=1 && 0<=z13<=1 && 0<=z23<=1 && 0<=z24<=1},\
     {z1, z2, z3, z4, z14, z13, z23, z24}]
\end{verbatim}
}}

\noindent
obtaining that
\[
F(\mathbf{z}) \le F(\zeta, \zeta, \zeta, \zeta, 1, 1, 1, 1) \sim 0.970247,
\]
where $\zeta \sim 0.695347 $.

\end{document}